\documentclass[a4paper,10pt]{amsart}

  \usepackage{amsmath}		   			                
  \usepackage{amsthm}					                
  \usepackage{amssymb}  				    	        
  \usepackage{mathrsfs}					                
  \usepackage[all]{xy}  	  		                    
  \usepackage{fancyhdr}                             	
  \usepackage{cases}  
  \usepackage{hyperref}
  
    \theoremstyle{remark}         \newtheorem{remark}{Remark}[section]
    \theoremstyle{remark} 		\newtheorem{example}{Example}[section]
    \theoremstyle{definition}		\newtheorem{definition}{Definition}[section]
    \theoremstyle{plain}			\newtheorem{theorem}{Theorem}[section]
    \theoremstyle{plain}			\newtheorem{lemma}[theorem]{Lemma}
    \theoremstyle{plain}			\newtheorem{corollary}[theorem]{Corollary}
    \theoremstyle{plain}			\newtheorem{proposition}[theorem]{Proposition}

\allowdisplaybreaks
\numberwithin{equation}{section}

\newcommand{\difl}{{}^\si\!\diffd}

\newcommand{\sbu}{{\scriptscriptstyle\bullet}}

\newcommand{\JJ}{\mbf{J}}
\newcommand{\cone}{\mathsf{cone}}

\newcommand{\q}{\mathrm{q}}
\newcommand{\tw}{\mathsf{tw}}

\DeclareMathOperator{\im}{im}				        	    
\DeclareMathOperator{\Hom}{Hom}				            
\DeclareMathOperator{\Ext}{Ext}				            
\DeclareMathOperator{\Aut}{Aut}				            
\DeclareMathOperator{\id}{id}				                
\DeclareMathOperator{\cha}{char}                          

\DeclareMathOperator{\gr}{gr}

\newcommand{\varphantom}[1]{\mathrel{\phantom{#1}}}
\newcommand{\pl}{\partial}

\newcommand{\ot}{\otimes}
\newcommand{\diffd}{\mathsf{d}}			             	

  \newcommand{\al}{\alpha}

  \newcommand{\si}{\sigma}

  \newcommand{\de}{\delta}
  \newcommand{\De}{\Delta}
  \newcommand{\lam}{\lambda}
  \newcommand{\Lam}{\Lambda}
  \newcommand{\vphi}{\varphi}

  \newcommand{\mbb}[1]{\mathbb{#1}}			    	
  \newcommand{\mbf}[1]{\mathbf{#1}}			    	
  \newcommand{\mc}[1]{\mathcal{#1}}			    	
  \newcommand{\mf}[1]{\mathfrak{#1}}				
  \newcommand{\kk}{\Bbbk} 			    	        
  \newcommand{\nan}{\mathbb{N}}			    	    
  \newcommand{\inn}{\mathbb{Z}} 				    

\begin{document}
	
	\title[Nakayama automorphisms of Ore extensions]{Nakayama automorphisms of Ore extensions over polynomial algebras}
	
	\author[Liyu Liu]{Liyu Liu}
	\address{School of Mathematical Science, Yangzhou University, No.\ 180 Siwangting Road, 225002 Yangzhou, Jiangsu, China}
	\email{lyliu@yzu.edu.cn}
	
	\author[Wen Ma]{Wen Ma}
	\address{School of Mathematical Science, Yangzhou University, No.\ 180 Siwangting Road, 225002 Yangzhou, Jiangsu, China}
	\email{2922117517@qq.com}

\begin{abstract}
	Nakayama automorphisms play an important role in several mathematical branches, which are known to be tough to compute in general. We compute the Nakayama automorphism $\nu$ of any Ore extension $R[x;\si,\de]$ over a polynomial algebra $R$ in $n$ variables for an arbitrary $n$. The formula of $\nu$ is obtained explicitly. When $\si$ is not the identity map, the invariant $E^G$ is also investigated in term of Zhang's twist, where $G$ is a cyclic group sharing the same order with $\si$.
\end{abstract}
\maketitle

\section{Introduction}

In the past three decades, a great deal of research appears on Artin--Schelter regular algebras arising from noncommutative projective algebraic geometry and on noetherian Hopf algebras. Associated to these algebras, there are automorphisms that play an important role in studying and/or classifying them. They are related with the study of rigid dualizing complexes, Hopf algebra actions, noncommutative invariant theory, Zariski cancellation, and so forth. Such automorphisms are nowadays called Nakayama automorphisms in the literature. Examples of algebras that have Nakayama automorphisms are: noetherian Artin--Schelter Gorenstein algebras, many noetherian Hopf algebras, some (co)invariant subalgebras of Artin--Schelter regular algebras under Hopf algebra actions, Poincar\'e--Birkhoff--Witt deformations of some graded algebras.  In general, Nakayama automorphisms are known to be tough to compute. We refer to \cite{Bell-Zhang:Zariski-cancellation}, \cite{Brown-Gilmartin:qhs-connectedHopf}, \cite{Brown-Zhang:AS-Gorenstein-Hopf}, \cite{Krahmer:qh-space}, \cite{Lv-Mao-Zhang:NAK}, \cite{Lv-Mao-Zhang:Nak-graded}, \cite{Reyes-Rogalski-Zhang:skew-CY-homo-identity}, \cite{Shen-Lu:Nak-PBW}, \cite{Yekutieli:dual-comp}, \cite{Yekutieli:rigid-dualizing-complex-universal-enveloping-algebra} and the references therein for the progress on this topic during the past years. 

Nakayama automorphisms, which turn out to be unique up to inner for an algebra, are related with twisted Calabi--Yau algebras (or skew Calabi--Yau algebras); in particular, a homologically smooth algebra whose Nakayama automorphism is inner is Calabi--Yau in the sense of Ginzburg \cite{Ginzburg:CY-alg}. Calabi--Yau algebra is an algebraic structure arising from the geometry of Calabi--Yau manifolds and homological mirror symmetry. It has attracted much interest in recent years.

On top of the examples we mentioned above, one can construct new algebras belonging to the class via a given algebra having a Nakayama automorphism; for instance, by Ore extensions. In order to study right coideal subalgebras of quantized enveloping algebras and 5-dimensional Artin--Schelter regular algebras, Wang, Wu and the first author proved that if $A$ is twisted Calabi--Yau with Nakayama automorphism $\nu_A$ then the Ore extension $E=A[x;\si,\de]$ has Nakayama automorphism $\nu_E$ such that $\nu_E|_A=\si^{-1}\nu_A$ and $\nu_E(x)=\lam x+b$ for an invertible $\lam \in A$ and $b\in A$ \cite{L-Wu-Wang:twisted-CY-Ore-extension}. The only drawback is that $\lam$, $b$ are not completely determined. This result was improved by several professionals in different contexts, cf \cite{Goodman-Krahmer:untwist-cy}, \cite{He-Oystaeyen-Zhang:Koz-AS-reg}, \cite{Zhu-Oystaeyen-Zhang:double-ore-extension}, etc. Among them, Zhu, Van Oystaeyen and Zhang showed that $\lam$ is equal to the homological determinant $\mathrm{hdet}(\si)$ of $\si$, introduced by J\o{}rgensen and Zhang \cite{Jorgensen-Zhang:gorenstein}, for any trimmed graded Ore extension $A[x;\si]$ over a Koszul Artin--Schelter regular algebra $A$. Notice that there are numerous automorphisms that are not graded for a specific graded algebra, even for a polynomial algebra. Besides, the parameter $b$ is still unknown for non-trimmed Ore extensions.

In this paper, we partially make up for the deficiency of the result in \cite{Zhu-Oystaeyen-Zhang:double-ore-extension}. We focus on Ore extensions $R[x;\si,\de]$ where $R$ is the polynomial algebra $\kk[z_1,\dots,z_n]$, but we do not impose any restrictive conditions about the automorphism $\si$ and the $\si$-derivation $\de$. By regarding $R$ as a noncommutative algebra, we introduce in \S\ref{sec:preliminaries} a noncommutative version of partial derivation $\De_p$ with respect to each variable $z_p$ ($1\leq p\leq n$). In terms of $\De_p$, there is a determinant $\JJ$ which is analogous to the classical Jacobian determinant. We call $\JJ$ the noncommutative Jacobian determinant, and discuss relations between it and the classical one. Following the idea in \cite{L-Wu-Wang:twisted-CY-Ore-extension}, we construct a bounded resolution of $E=R[x;\si,\de]$ by finitely generated free left $E^e$-modules in \S\ref{sec:free-resol}. The differentials of this resolution are expressed explicitly, by using $\De_p$ and $\JJ$.

In \S\ref{sec:naka-auto}, we compute the Nakayama automorphism $\nu$ of $E$, completely determining the two parameters $\lam$ and $b$. We capture $\lam$ and $b$ in several different cases. The results are summarized as follows (Theorems \ref{thm:main-thm-diff}, \ref{thm:main-thm-general}, and  Corollary \ref{cor:CY-nece-suff}):

\begin{theorem}
	Let $E=R[x;\si,\de]$ be an Ore extension, and $J$ the Jacobian determinant of $\si$. Then the Nakayama automorphism $\nu$ satisfies $\nu|_R=\si^{-1}$ and
	\begin{enumerate}
		\item $\nu(x)=x+\nabla\cdot\mc{X}_\de$ if $\si=\id$, where $\mc{X}_\de$ is a vector field determined by $\de$ and $\nabla\cdot\mc{X}_\de$ stands for the divergence of $\mc{X}_\de$.
		\item $\nu(x)=Jx+J\kappa-\si_\q^{-1}(\kappa)$ if $\si\neq \id$, where $\kappa$ is in the quotient field of $R$ and $\si_\q$ is an extension of $\si$. Both $\kappa$ and $\si_\q$ are uniquely determined.
	\end{enumerate}
	As a consequence, $E$ is Calabi--Yau if and only if $\si=\id$ and $\nabla\cdot\mc{X}_\de=0$.
\end{theorem}

Suppose $\si\neq \id$, and let $G$ be a cyclic group of order equal to that of $\si$ (also equal to that of $\nu$). There is a $G$-action on $E$ naturally. We investigate the invariant $E^G$ in \S\ref{sec:inv}. In this section, we first extend $\si$, $\de$ to the quotient field $R_\q$ of $R$. The extended automorphism is denoted by $\si_\q$, and the extended $\si_\q$-derivation is denoted by $\de_\q$, respectively. A larger Ore extension $E_\q=R_\q[x; \si_\q, \de_\q]$ is hence obtained. Next we show that the invariant $E_\q^G$ is isomorphic to the Zhang's twist $\tw(\widehat{R_\q}, \widehat{\si_\q})$ of a graded algebra $\widehat{R_\q}$ via a graded automorphism $\widehat{\si_\q}$. Later on, we discuss the subalgebra $E^G$ of $E_\q^G$ in two cases. The case $\kappa\in R$ is very easy, under which we deduce that $E^G$ is isomorphic to the Zhang's twist $\tw(\widehat{R}, \widehat{\si})$; the other case $\kappa\notin R$ is quite subtle because sometimes $E^G$ carries a graded algebra structure in a natural way but sometimes not. If $E^G$ is not graded, we equip it with a filtration, whose associated graded algebra $\gr E^G$ is proven to be a graded subalgebra of $\tw(\widehat{R}, \widehat{\si})$. The graded subalgebra may happen to be $\tw(\widehat{R}, \widehat{\si})$ itself. The results are summarized as follows (Theorem \ref{thm:gr-invariant}):

\begin{theorem}
	Let $E=R[x;\si,\de]$ be an Ore extension with $\si\neq\id$. Then
	\begin{enumerate}
		\item $\gr E^G$ is isomorphic to a graded subalgebra of $\tw(\widehat{R}, \widehat{\si})$,
		\item $\gr E^G$ is isomorphic to $\tw(\widehat{R}, \widehat{\si})$ if either $\nu(x)=Jx$ or $\si$ is of finite order $r$ with $\cha\kk\nmid r$.
	\end{enumerate}
\end{theorem}

Throughout this paper, $\kk$ is a field and all algebras are over $\kk$ unless stated otherwise. Unadorned $\ot$ means $\ot_{\kk}$. 

\section{Preliminaries}\label{sec:preliminaries}
Let $A$ be an algebra and $M$ an $A$-bimodule. The group of algebra automorphisms of $A$ is denoted by $\Aut(A)$. For any $\si\in \Aut(A)$, denote by ${}^\si M$ (resp.\ $M^\si$) the left $A$-module (resp.\ right $A$-module) whose ground $\kk$-module is the same with $M$ and whose left (resp.\ right) $A$-action is twisted by $\si$, that is, $a\triangleright  m= \si(a)m$ (resp.\ $m\triangleleft a=m\si(a)$) for any $a\in A$, $m\in M$. 

Let $A^\mathrm{op}$ be the opposite algebra of $A$, and $A^e = A\ot A^\mathrm{op}$ the enveloping algebra of $A$. We will often identify $A$-bimodules with left or right $A^e$-modules, for convenience. Recall that $A$ is \textit{homologically smooth} if $A$ as a left (or equivalently, right) $A^e$-module, admits a finitely generated projective resolution of finite length.

\begin{definition}
	An algebra $A$ is called $\nu$-\textit{twisted Calabi--Yau} of dimension $d$ for $\nu\in\Aut(A)$ and $d\in\nan$ if
	\begin{enumerate}
		\item $A$ is homologically smooth,
		\item there are isomorphisms of $A$-bimodules
		\[
		\Ext^i_{A^e}(A,A^e)\cong \begin{cases}
		A^\nu, & i=d, \\ 0, & i\neq d,
		\end{cases}
		\]
		in which the regular left module structure on $A^e$ is used for computing the $\Ext$-group, and the right one induces the $A$-bimodule structure on the $\Ext$-group.
	\end{enumerate}
\end{definition}

Let $R$ be an algebra. Given an endomorphism $\si$ on $R$ and a $\si$-derivation $\de$, one obtains an extension $E=R[x;\si,\de]$ over $R$ which is called an Ore extension. We assume that the reader is familiar with this topic. For details we refer to \cite{McConnell-Robson:noncomm-ring}. Since $E$ will inherit nice properties from $R$ if $\si$ is an automorphism, we require in this paper that $\si$ is always an automorphism when Ore extensions are mentioned. Customarily, an Ore extension is called \textit{trimmed} if $\de=0$, called \textit{differential} if $\si=\id$. In both cases, we write $R[x;\si]$ and $R[x;\de]$ respectively.

\begin{lemma}\label{lem:ses-ore}
	Let $E=R[x;\si,\de]$ be an Ore extension. Then there is a short exact sequence
	\begin{equation*}
	0\xrightarrow{\quad} E^\si\ot_RE \xrightarrow[\quad]{\bar{\rho}} E\ot_RE \xrightarrow{\quad} E \xrightarrow{\quad}  0
	\end{equation*}
	of $E$-bimodules where $\bar{\rho}(1\ot 1)=x\ot 1-1\ot x$.
\end{lemma}

\begin{theorem}\cite[Thm.\ 2]{L-Wu-Wang:twisted-CY-Ore-extension}\label{thm:Ore-LWW}
	Let $E=R[x;\si,\de]$ be an Ore extension. If $R$ is $\nu_R$-twisted Calabi--Yau of dimension $d$, then $E$ is $\nu_E$-twisted Calabi--Yau of dimension $d+1$ where $\nu_E$ is determined by $\nu_E|_R=\si^{-1}\nu_R$ and $\nu_E(x)=\lam x+b$ with $\lam$, $b$ in $R$ and $\lam$ invertible. Moreover, $\lam$ depends only on $\si$, and $b$ depends on both $\si$ and $\de$.
\end{theorem}

\begin{remark}\label{rmk:lambda-sigma}
	The parameters $\lam$ and $b$ have not been completely determined so far. Though, other professionals improve Theorem \ref{thm:Ore-LWW} a lot, mainly dealing with the first parameter $\lam$ in several contexts. As far as the author knows, the following two are of importance and of interest:
	\begin{enumerate}
		\item \cite{Goodman-Krahmer:untwist-cy} If $\si=\nu_R$ and $\de=0$, then $\nu_E(x)=x$. As a consequence, every twisted Calabi--Yau algebra admits a trimmed Ore extension that is Calabi--Yau.
		\item \cite{Zhu-Oystaeyen-Zhang:double-ore-extension} If $R$ is graded Koszul, $\si\in \operatorname{GrAut}(R)$, and $\de=0$, then $\nu(x)=\mathrm{hdet}(\si)x$.
	\end{enumerate}
\end{remark}

From now on, let $R=\kk[z_1,z_2,\ldots,z_n]$ be the polynomial algebra in $n$ variables. Denote $\si_i=\si(z_i)$ and $\de_i=\de(z_i)$. Then $\si$ is completely determined by the $n$-tuple $(\si_1,\ldots,\si_n)$, and respectively, $\de$ is determined by $(\de_1,\ldots,\de_n)$. 

For any polynomial $\vphi$ in $R$, we write $\diffd \vphi=1\ot \vphi-\vphi\ot 1$; it is called the noncommutative 1-form of $\vphi$ in noncommutative geometry. For each $1\leq p\leq n$, define the linear map $\De_p\colon R\to R\ot R$ by
\[
\De_p(z_1^{i_1}\cdots z_n^{i_n})=\sum_{s=1}^{i_p}z_p^{i_p-s}z_{p+1}^{i_{p+1}}\cdots z_n^{i_n}\ot z_1^{i_1}\cdots z_{p-1}^{i_{p-1}}z_p^{s-1}.
\]

\begin{lemma}\label{lem:noncomm-deriv}
	For any polynomial $\vphi$, one has $\diffd \vphi=\sum\limits_{p=1}^{n}\De_p(\vphi)\diffd z_p$.
\end{lemma}

\begin{proof}
	Straightforward.
\end{proof}

If $\mu$ is the multiplication map of $R$, then we can check $\mu\De_p=\pl/\pl z_p$. Hence $\De_p$ is sometimes called the noncommutative partial derivation with respect to $z_p$, and Lemma \ref{lem:noncomm-deriv} is analogous to the total differential formula in calculus. 

We will introduce a determinant $\JJ$ in terms of $\De_p$ which are somewhat Jacobian-like. By definition, $\JJ=|\De_j(\si_i)|_{n\times n}$, and we called it the noncommutative determinant of $\si$. In fact, $\mu(\JJ)$ is equal to the common Jacobian determinant $J$ of $\si$. For any pair of sequences $(i_1,\ldots,i_p)$ and $(j_1,\ldots,j_p)$ in $\{1,2,\ldots,n\}$, define $\JJ_{j_1\ldots j_p}^{i_1\ldots i_p}$ to be the determinant of order $p$ whose $(u,v)$-entry is $\De_{j_v}(\si_{i_u})$. 


\section{Free resolutions}\label{sec:free-resol}
In this section, we follow the idea used in \cite{L-Wu-Wang:twisted-CY-Ore-extension} to construct a free resolution of $E$ as an $E^e$-module. 

Let $\mc{K}^R_{\sbu}$ be the Koszul resolution of $R$, i.e.,  $\mc{K}^R_p$ is free over $R\ot R$ with basis $\{\mf{e}_{i_1\ldots i_p}\mid 1\leq i _1<\cdots< i_p\leq n\}$. Then $\mc{K}_{\sbu}:=E\ot_R\mc{K}^R_{\sbu}\ot_RE$ is a left $E^e$-projective resolution of $E\ot_RE$ (here we identify $E$-bimodules with left $E^e$-modules) whose differentials are given by 
\[
d_p(\mf{e}_{i_1\ldots i_p})=\sum_{v=1}^p(-1)^{v}\diffd z_{i_v}\mf{e}_{i_1\ldots\widehat{i_v}\ldots i_p}.
\] 
Likewise, denoting by $\mf{e}'_{i_1\ldots i_p}$ a copy of $\mf{e}_{i_1\ldots i_p}$, we then have an $E^e$-projective resolution $\mc{K}'_{\sbu}$ of $E^\si\ot_RE$ whose differentials are given by 
\[
d_p'(\mf{e}'_{i_1\ldots i_p})=\sum_{v=1}^p(-1)^{v}\difl z_{i_v}\mf{e}'_{i_1\ldots\widehat{i_v}\ldots i_p}, 
\]
where $\difl z_i=1\ot z_i-\si_i\ot 1$. Let us lift $\bar{\rho}$ to a morphism $\rho\colon \mc{K}'_{\sbu}\to\mc{K}_{\sbu}$ of complexes for trimmed and differential Ore extensions respectively.

\begin{proposition}\label{prop:lifting-trimmed}
	Let $E=R[x;\si]$ be a trimmed Ore extension, and $\bar{\rho}$ as given in Lemma \ref{lem:ses-ore}. Define $\rho=\{\rho_p\colon \mc{K}'_{p}\to\mc{K}_{p}\}_{0\leq p\leq n}$ to be
	\[
	\rho_p(\mf{e}'_{i_1\ldots i_p})=(x\ot 1)\mf{e}_{i_1\ldots i_p}-(1\ot x)\sum_{j_1<\dots< j_p}\JJ^{i_1\ldots i_p}_{j_1\ldots j_p}\mf{e}_{j_1\ldots j_p}.
	\]
	Then $\rho$ is a lifting of $\bar{\rho}$.
\end{proposition}

\begin{proof}
	Let us compute $\rho_{p-1}d'_p(\mf{e}'_{i_1\ldots i_p})$ and $d_p\rho_p(\mf{e}'_{i_1\ldots i_p})$ for all $p$. We have
	\begin{align*}
	&\varphantom{=}\rho_{p-1}d'_p(\mf{e}'_{i_1\ldots i_p})=\rho_{p-1}\biggl(\sum_{v=1}^{p}(-1)^{v}\difl z_{i_v}\mf{e}'_{i_1\ldots\widehat{i_v}\ldots i_p}\biggr)\\
	&=\sum_{v=1}^{p}(-1)^{v}\difl z_{i_v}\biggl((x\ot 1)\mf{e}_{i_1\ldots\widehat{i_v}\ldots i_p}-\sum_{j_1\dots j_{p-1}}(-1)^{v}(1\ot x)\JJ^{i_1\ldots\widehat{i_v}\ldots i_p}_{j_1\ldots j_{p-1}}\mf{e}_{j_1\ldots j_{p-1}}\biggr)\\
	&=(x\ot 1)\sum_{v=1}^{p}(-1)^{v}\diffd z_{i_v}\mf{e}_{i_1\ldots\widehat{i_v}\ldots i_p}-(1\ot x)\sum_{v=1}^{p}\sum_{j_1\dots j_{p-1}}(-1)^{v}\diffd \si_{i_v}\JJ^{i_1\ldots\widehat{i_v}\ldots i_p}_{j_1\ldots j_{p-1}}\mf{e}_{j_1\ldots j_{p-1}},
	\end{align*}
	and
	\begin{align*}
	&\varphantom{=}d_p\rho_p(\mf{e}'_{i_1\ldots i_p})=d_p\biggl((x\ot 1)\mf{e}_{i_1\ldots i_p}-(1\ot x)\sum_{j_1\dots j_p}\JJ^{i_1\ldots i_p}_{j_1\ldots j_p}\mf{e}_{j_1\ldots j_p}\biggr)\\
	&=(x\ot 1)\sum_{v=1}^{p}(-1)^{v}\diffd z_{i_v}\mf{e}_{i_1\ldots\widehat{i_v}\ldots i_p}-(1\ot x)\sum_{v=1}^{p}\sum_{j_1\dots j_p}(-1)^{v}\diffd z_{j_v}\JJ^{i_1\ldots i_p}_{j_1\ldots j_p}\mf{e}_{j_1\ldots\widehat{j_v}\ldots j_p}.
	\end{align*}
	In order to prove that they are equal, it suffices to show
	\begin{equation}\label{eq:eq1}
	\sum_{v=1}^{p}\sum_{j_1\dots j_{p-1}}(-1)^{v}\diffd \si_{i_v}\JJ^{i_1\ldots\widehat{i_v}\ldots i_p}_{j_1\ldots j_{p-1}}\mf{e}_{j_1\ldots j_{p-1}}=\sum_{v=1}^{p}\sum_{j_1\dots j_p}(-1)^{v}\diffd z_{j_v}\JJ^{i_1\ldots i_p}_{j_1\ldots j_p}\mf{e}_{j_1\ldots\widehat{j_v}\ldots j_p}.
	\end{equation}
	Notice that by Lemma \ref{lem:noncomm-deriv}, we have
	\begin{align*}
	\sum_{v=1}^{p}(-1)^{v}\diffd \si_{i_v}\JJ^{i_1\ldots\widehat{i_v}\ldots i_p}_{j_1\ldots j_{p-1}}=\sum_{l=1}^{n}\sum_{v=1}^{p}(-1)^{v}\De_l(\si_{i_v})\diffd z_l\JJ^{i_1\ldots\widehat{i_v}\ldots i_p}_{j_1\ldots j_{p-1}}=-\sum_{l=1}^{n}\diffd z_l\JJ^{i_1\ldots i_p}_{lj_1\ldots j_{p-1}},
	\end{align*}
	where the last equality holds by expanding the determinant $\JJ^{i_1\ldots i_p}_{lj_1\ldots j_{p-1}}$ along the first column.
	Since $\JJ^{i_1\ldots i_p}_{lj_1\ldots j_{p-1}}=0$ for all $l\in\{j_1,\ldots,j_{p-1}\}$, we have
	\begin{align*}
	\text{LHS of }\eqref{eq:eq1}&=-\sum_{j_1\dots j_{p-1}}\sum_{l=1}^{n}\diffd z_l\JJ^{i_1\ldots i_p}_{lj_1\ldots j_{p-1}}\mf{e}_{j_1\ldots j_{p-1}}\\
	&=-\sum_{lj_1\dots j_{p-1}}\diffd z_l\JJ^{i_1\ldots i_p}_{lj_1\ldots j_{p-1}}\mf{e}_{\widehat{l}j_1\ldots j_{p-1}}+\sum_{j_1l\dots j_{p-1}}\diffd z_l\JJ^{i_1\ldots i_p}_{j_1l\ldots j_{p-1}}\mf{e}_{j_1\widehat{l}\ldots j_{p-1}}\\
	&\varphantom{=}{}+\cdots+\sum_{j_1\dots j_{p-1}l}(-1)^{p}\diffd z_l\JJ^{i_1\ldots i_p}_{j_1\ldots j_{p-1}l}\mf{e}_{j_1\ldots j_{p-1}\widehat{l}}\\
	&=\text{RHS of }\eqref{eq:eq1}. \qedhere
	\end{align*}
\end{proof}

\begin{proposition}\label{prop:lifting-diff}
	Let $E=R[x;\de]$ be a differential Ore extension, and $\bar{\rho}$ as given in Lemma \ref{lem:ses-ore}. Define $\rho=\{\rho_p\colon \mc{K}'_{p}\to\mc{K}_{p}\}_{0\leq p\leq n}$ to be
	\[
	\rho_p(\mf{e}'_{i_1\ldots i_p})=(x\ot 1-1\ot x)\mf{e}_{i_1\ldots i_p}+\sum_{s=1}^{n}\sum_{u=1}^{p}(-1)^u\De_{s}(\de_{i_u})\mf{e}_{si_1\ldots \widehat{i_u}\dots i_p}.
	\]
	Then $\rho$ is a lifting of $\bar{\rho}$.
\end{proposition}

\begin{proof}
Like in the proof of Proposition \ref{prop:lifting-diff}, we have
	\begin{align*}
	&\varphantom{=}\rho_{p-1}d'_p(\mf{e}'_{i_1\ldots i_p})=\rho_{p-1}\biggl(\sum_{v=1}^{p}(-1)^{v}\diffd z_{i_v}\mf{e}'_{i_1\ldots\widehat{i_v}\ldots i_p}\biggr)\\
	&=\sum_{v=1}^{p}(-1)^{v}\diffd z_{i_v}\biggl((x\ot 1-1\ot x)\mf{e}_{i_1\ldots\widehat{i_v}\ldots i_p}+\sum_{s=1}^{n}\sum_{u=1}^{v-1}(-1)^u \De_s(\de_{i_u}) \mf{e}_{si_1\dots\widehat{i_u}\dots\widehat{i_v}\ldots i_p}\\
	&\varphantom{=}{}+\sum_{s=1}^{n}\sum_{u=v+1}^{p}(-1)^{u-1}\De_s(\de_{i_u}) \mf{e}_{si_1\dots\widehat{i_v}\ldots \widehat{i_u}\ldots i_p}\ \biggr)\\
	&=(x\ot 1-1\ot x)\sum_{v=1}^{p}(-1)^{v}\diffd z_{i_v}\mf{e}_{i_1\ldots\widehat{i_v}\ldots i_p}-\sum_{v=1}^{p}(-1)^{v} \diffd \de_{i_v}\mf{e}_{i_1\ldots\widehat{i_v}\ldots i_p}\\
	&\varphantom{=}{}+\sum_{s=1}^{n}\sum_{1\leq u< v\leq p}(-1)^{u+v}
	\begin{vmatrix}
	\De_s(\de_{i_u}) & \diffd z_{i_u} \\[.8ex] \De_s(\de_{i_v}) & \diffd z_{i_v}
	\end{vmatrix}
	\mf{e}_{si_1\ldots\widehat{i_u}\ldots\widehat{i_v}\ldots i_p},
	\end{align*}
	and by Lemma \ref{lem:noncomm-deriv},
	\begin{align*}
	&\varphantom{=}d_p\rho_p(\mf{e}'_{i_1\ldots i_p})=d_p\biggl((x\ot 1-1\ot x)\mf{e}_{i_1\ldots i_p}+\sum_{s=1}^{n}\sum_{u=1}^{p}(-1)^u\De_{s}(\de_{i_u})\mf{e}_{si_1\ldots \widehat{i_u}\dots i_p}\biggr)\\
	&=(x\ot 1-1\ot x)\sum_{v=1}^{p}(-1)^{v}\diffd z_{i_v}\mf{e}_{i_1\ldots\widehat{i_v}\ldots i_p}+\sum_{s=1}^{n}\sum_{u=1}^{p}(-1)^u\De_{s}(\de_{i_u})\biggl( -\diffd z_s \mf{e}_{i_1\ldots \widehat{i_u}\dots i_p}\\
	&\varphantom{=}{}+\sum_{v=1}^{u-1}(-1)^{v+1}\diffd z_{i_v} \mf{e}_{si_1\ldots\widehat{i_v}\dots \widehat{i_u}\dots i_p}+\sum_{v=u+1}^{p}(-1)^{v}\diffd z_{i_v} \mf{e}_{si_1\ldots\widehat{i_u}\dots \widehat{i_v}\dots i_p} \biggr)\\
	&=(x\ot 1-1\ot x)\sum_{v=1}^{p}(-1)^{v}\diffd z_{i_v}\mf{e}_{i_1\ldots\widehat{i_v}\ldots i_p}-\sum_{u=1}^{p}(-1)^u\diffd \de_{i_u}\mf{e}_{i_1\ldots \widehat{i_u}\dots i_p}\\
	&\varphantom{=}{}+\sum_{s=1}^{n}\sum_{1\leq u< v\leq p}(-1)^{u+v}
	\begin{vmatrix}
	\De_s(\de_{i_u}) & \diffd z_{i_u} \\[.8ex] \De_s(\de_{i_v}) & \diffd z_{i_v}
	\end{vmatrix}
	\mf{e}_{si_1\ldots\widehat{i_u}\ldots\widehat{i_v}\ldots i_p}.
	\end{align*}
	Therefore, $\rho_{p-1}d'_p=d_p\rho_p$ holds true for all $p$.
\end{proof}

\begin{corollary}\label{cor:cone-resolu}
	Let $E$ and $\rho$ be in Propositions \ref{prop:lifting-trimmed} or \ref{prop:lifting-diff}, $\cone(\rho)$ is a free resolution of $E$ as a left $E^e$-module.
\end{corollary}

\section{Nakayama automorphisms}\label{sec:naka-auto}

In order to compute the Nakayama automorphism $\nu$ of trimmed or differential $E$ (or equivalently the cohomological group $\Ext^{n+1}_{E^e}(E, E^e)$), let us observe the differential $d^\rho_{n+1}\colon\cone(\rho)_{n+1}\to \cone(\rho)_n$. By Corollary \ref{cor:cone-resolu}, we know
\[
\Ext^{n+1}_{E^e}(E, E^e)=H^{n+1}\Hom_{E^e}(\cone(\rho), E^e)=E^e/\im d_{n+1}^{\rho\vee}
\]
where $d_{n+1}^{\rho\vee}=\Hom_{E^e}(d_{n+1}^{\rho}, E^e)$ is the dual of $d_{n+1}^{\rho}$. The $E$-bimodule structure on $E^e/\im d_{n+1}^{\rho\vee}$ inherits from the multiplication on $E^e$ from right, i.e., 
\[
e_1\triangleright \Bigl(\sum x_1\ot x_2+\im d_{n+1}^{\rho\vee}\Bigr)\triangleleft e_2=\sum x_1e_2\ot e_1x_2+\im d_{n+1}^{\rho\vee}.
\]
On the other hand, $\Ext^{n+1}_{E^e}(E, E^e)\cong E^\nu$ is proven in Theorem \ref{thm:Ore-LWW}.

\begin{lemma}\label{lem:nak-tensor}
Let $E$ be trimmed or differential whose Nakayama automorphism is $\nu$. For any $a$, $f$ in $E$, $\nu(a)=f$ if and only if $a\ot 1-1\ot f \in \im d_{n+1}^{\rho\vee}$. 
\end{lemma}

\begin{proof}
We choose an automorphism $\theta\colon E^e/\im d_{n+1}^{\rho\vee}\to E^\nu$, and suppose $\theta^{-1}(1)=\sum \overline{y_1\ot y_2}:=\sum y_1\ot y_2+\im d_{n+1}^{\rho\vee}$ and $\theta(\overline{1\ot 1})=k$. It follows that
\[
1=\theta\Bigl(\sum \overline{y_1\ot y_2}\Bigr)=\sum y_2\triangleright\theta(\overline{1\ot 1})\triangleleft y_1=\sum y_2k\nu(y_1)=k\sum y_2\nu(y_1).
\]
Necessarily, $k\in\kk\setminus\{0\}$, and without loss of generality, we assume $k=1$. Hence
\begin{align*}
 \nu(a)=f & \iff \theta(\overline{1\ot 1})\vartriangleleft a=f\vartriangleright \theta(\overline{1\ot 1}) \\
 & \iff \theta(\overline{a\ot 1})=\theta(\overline{1\ot f}) \\
 &\iff \overline{a\ot 1}=\overline{1\ot f} \\
 &\iff a\ot 1-1\ot f \in \im d_{n+1}^{\rho\vee}. \qedhere
\end{align*}
\end{proof}

\subsection{Trimmed case}\label{subsec:trimmed}
As free $E^e$-modules, $\cone(\rho)_{n+1}$ and $\cone(\rho)_{n}$ are of rank 1 and $n$ respectively. By Proposition \ref{prop:lifting-trimmed},
\begin{align*}
d^\rho_{n+1}(\mf{e}'_{12\ldots n})&=\rho_n(\mf{e}'_{12\ldots n})+d'_n(\mf{e}'_{12\ldots n})\\
&=\bigl( x\ot 1-(1\ot x)\JJ\bigr) \mf{e}_{12\ldots n}+\sum_{i=1}^{n}(-1)^i\difl z_i \mf{e}'_{1\ldots\hat{i}\ldots n}. 
\end{align*}
The differential is thus represented by the matrix (multiplying from right)
\[
M=\bigl( x\ot 1-(1\ot x)\JJ, -\difl z_1, \difl z_2, \ldots, (-1)^n\difl z_n \bigr).
\]
We then immediately have $f\ot1-1\ot \si^{-1}(f)\in\im d_{n+1}^{\rho\vee}$ for all $f\in R$. It follows from Lemma \ref{lem:nak-tensor} that $\nu(f)=\si^{-1}(f)$, in accordance with Theorem \ref{thm:Ore-LWW}. For $\nu(x)$, we write $\JJ=\sum g_1\ot g_2$ and recall that the Jacobian determinant $J$ of $\si$ satisfies $J=\mu(\JJ)=\sum g_1g_2$. Hence
\begin{align*}
x\ot 1 &\equiv (1\ot x)\JJ=\sum g_1\ot g_2x \\
&\equiv \sum 1\ot g_2x\si^{-1}(g_1)=\sum 1\ot g_2g_1x \\
&=1\ot Jx  \pmod {\im d_{n+1}^{\rho\vee}}.
\end{align*}
By Lemma \ref{lem:nak-tensor}, $\nu(x)=Jx$. 

\begin{remark}
	When $\si$ is a graded algebra homomorphism,  $J$ is equal to the homological determinant $\mathrm{hdet}(\si)$ (cf.\ \cite{Jorgensen-Zhang:gorenstein}). Hence, the formula $\nu(x)=Jx$ coincides with \cite[Prop.\ 3.15]{Zhu-Oystaeyen-Zhang:double-ore-extension}.
\end{remark}

\subsection{Differential case}
In this case, by Proposition \ref{prop:lifting-diff}, 
\begin{align*}
d^\rho_{n+1}(\mf{e}'_{12\ldots n})&=\rho_n(\mf{e}'_{12\ldots n})+d'_n(\mf{e}'_{12\ldots n})\\
&=\biggl( x\ot 1-1\ot x-\sum_{u=1}^{n}\De_{u}(\de_{u})\biggr) \mf{e}_{12\ldots n}+\sum_{i=1}^{n}(-1)^i\diffd z_i \mf{e}'_{1\ldots\hat{i}\ldots n},
\end{align*}
and thus
\[
M=\biggl( x\ot 1-1\ot x-\sum_{u=1}^{n}\De_u(\de_u), -\diffd z_1, \diffd z_2, \ldots, (-1)^n\diffd z_n \biggr).
\] 
Similar with the trimmed case, we have $\nu(f)=f$ for any polynomial $f\in R$,  in accordance with Theorem \ref{thm:Ore-LWW}.

Since in the differential case
\[
\de=\sum_{i=1}^{n}\de_i\frac{\pl}{\pl z_i},
\]
we then have the vector field $\mc{X}_\de=(\de_1,\de_2,\dots,\de_n)$ on the affine space $\mbb{A}^n_\kk$, and then $\nabla\cdot\mc{X}_\de$ denotes the divergence of this vector field. 

\begin{theorem}\label{thm:main-thm-diff}
	Let $E=R[x;\de]$ be a differential Ore extension. The Nakayama automorphism $\nu$ of $E$ is then given by $\nu|_R=\id$ and
	\[
	\nu(x)=x+\nabla\cdot\mc{X}_\de.
	\]
\end{theorem}

\begin{proof}
	It is sufficient to show the expression of $\nu(x)$ given in the theorem is true. In fact, let us write $\sum f_1\ot f_2$ for $\sum_{u=1}^{n}\De_u(\de_u)$. Then
	\begin{align*}
	x\ot 1 &\equiv 1\ot x+\sum f_1\ot f_2 \\
	&=1\ot x-\sum \diffd f_1(1\ot f_2)+1\ot \sum f_1f_2 \\
	&\equiv 1\ot x+\sum_{u=1}^{n}1\ot \mu(\De_u(\de_u)) \\
	&=1\ot \biggl( x+\sum_{u=1}^{n} \frac{\pl\de_u}{\pl z_u} \biggr) \\
	&=1\ot (x+\nabla\cdot \mc{X}_\de) \pmod {\im d_{n+1}^{\rho\vee}}. \qedhere
	\end{align*}
\end{proof}

\begin{corollary}\label{cor:CY-nece-suff}
	Let $E=R[x;\si,\de]$ be an Ore extension. Then $E$ is Calabi--Yau if and only if $\si=\id$ and $\nabla\cdot\mc{X}_\de=0$.
\end{corollary}

\begin{proof}
	The necessity follows from Theorem \ref{thm:Ore-LWW}, and the sufficiency from Theorem \ref{thm:main-thm-diff}.
\end{proof}

\subsection{Non-differential case}

We are going to compute the Nakayama automorphism of $E=R[x;\si,\de]$ with $\si\neq \id$. First of all, let us present a lemma about $\si$-derivations.

\begin{lemma}\label{lem:skew-derivation}
	Let $\si\neq\id$ and $\de$ be a $\si$-derivation. Then there is a unique $\kappa$ in the quotient field $R_\q$ of $R$ such that $\de(h)=\kappa(\si(h)-h)$ for all $h\in R$.
\end{lemma}

\begin{proof}
	Since $R$ is commutative, by applying the Leibniz's rule to $\de(fg)$ and $\de(gf)$, we have
	\[
	\de(f)(\si(g)-g)=\de(g)(\si(f)-f)
	\]
	for all $f$, $g\in R$. Choose any $f$ such that $\si(f)\neq f$. If $\de(f)=0$, then $\de=0$, and we hence obtain $\kappa=0$. If $\de(f)\neq 0$, then $\de(g)\neq 0$ if and only if $\si(g)\neq g$. So the quotient $\de(f)/(\si(f)-f)$ is independent of the choice of $f$, which is the desired $\kappa$.
\end{proof}

Notice that the automorphism $\si$ on $R$ extends to $R_\q$ naturally, we write $\si_\q\in\Aut(R_\q)$ for the extension.

\begin{theorem}\label{thm:main-thm-general}
	Let $E=R[x;\si,\de]$ be an Ore extension with $\si\neq\id$, and $\kappa$ as in Lemma \ref{lem:skew-derivation}. Then the Nakayama automorphism $\nu$ of $E$ is given by $\nu|_R=\si^{-1}$ and
	\begin{equation}\label{eq:nak-general}
	\nu(x)=Jx+J\kappa-\si_\q^{-1}(\kappa).
	\end{equation}
\end{theorem}

\begin{proof}
	We have proven the theorem in \S\ref{subsec:trimmed} for $\de=0$. When $\de\neq 0$, we know that $\nu(x)=Jx+b$ for an undetermined $b\in R$ by Theorem \ref{thm:Ore-LWW}. Let $f\in R$ be such that $\si(f)\neq f$. We then apply $\nu$ to the relation $xf=\si(f)x+\de(f)$, obtaining that $(Jx+b)\si^{-1}(f)=f(Jx+b)+\si^{-1}\de(f)$, or equivalently, $J\de\si^{-1}(f)+b\si^{-1}(f)=fb+\si^{-1}\de(f)$. Hence
    \begin{align*}
    b&=J\frac{\de\si^{-1}(f)}{f-\si^{-1}(f)}-\frac{\si^{-1}\de(f)}{f-\si^{-1}(f)}\\
    &=J\kappa-\si_\q^{-1}\biggl(\frac{\de(f)}{\si(f)-f}\biggr)\\
    &=J\kappa-\si_\q^{-1}(\kappa),
    \end{align*}
    and it follows that $\nu(x)=Jx+J\kappa-\si_\q^{-1}(\kappa)$, as desired.
\end{proof}

\begin{remark}
	As a consequence, we have
	\begin{equation*}
	\si_\q^r(\kappa)\equiv J^{-r}\kappa \pmod R
	\end{equation*}
	for all integers $r$. This equation is trivial if $\kappa \in R$, however, this is not obvious if $\kappa \notin  R$. The authors have not yet found a proof of it without using Nakayama automorphism.
\end{remark}

\begin{corollary}\label{cor:order}
	The orders of $\si$ and of $\nu$ are equal.
\end{corollary}

\begin{proof}
	Denote by $o(\al)$ the order of an automorphism $\al$. Then $o(\si)\leq o(\nu)$ since $\nu|_R=\si^{-1}$. For $o(\si)\geq o(\nu)$, we may assume $o(\si)$ is finite, say $m$. It follows from \eqref{eq:nak-general} that $\nu^r(x)=J^rx+J^r\kappa-\si^{-r}(\kappa)$ for all integers $r$. Hence $\nu^m(x)=x$, forcing $o(\nu)\leq m$, as required.
\end{proof}

\section{Invariants under Nakayama automorphisms}\label{sec:inv}
In the present section, we suppose $\si\neq \id$, and let $\kappa$ be as in Lemma \ref{lem:skew-derivation}. Let $G$ be a cyclic group of order equal to that of $\si$, then by Corollary \ref{cor:order}, we have $G$-actions on $R$ and on $E$ naturally.

Our goal is to study the invariant $E^G=\{a\in E\mid \nu(a)=a\}$. For the purpose, let us construct a subsidiary Ore extension as follows. Since $\si_\q$ is an automorphism on $R_\q$, we define $\de_\q\colon R_\q\to R_\q$ by $\de_\q(c)=\kappa(\si_\q(c)-c)$ for all $c\in R_\q$. It is evident to check that $\de_\q$ is a $\si_\q$-derivation and $\de_\q|_R=\de$. So we have an Ore extension $E_\q=R_\q[x;\si_\q, \de_\q]$ that contains $E$ as a subalgebra.

The group $G$ acts on $E_\q$ via $\nu_\q$ where $\nu_\q\in\Aut(E_\q)$ is the unique extension of $\nu$. We will determine $E_\q^G$. Let $\mathfrak{S}^r_s(a_1, a_2, \dots, a_s)$ be the elementary symmetric polynomial of degree $r$ in variables $a_1, a_2, \dots, a_s$, where $1\leq r\leq s$. We adopt $\mathfrak{S}^0_s(a_1, a_2, \dots, a_s)=1$.

\begin{lemma}\label{lem:h-delta-power}
	For all $r\in\nan$, $(x+\kappa)^r=\sum\limits_{i=0}^{r} \mathfrak{S}_r^{r-i}(\kappa,\si_\q(\kappa),\ldots, \si^{r-1}_\q(\kappa))x^i$.
\end{lemma}

\begin{proof}
	Let us prove it by induction on $r$. When $r=0$, this is clearly true. Now let $r\geq 1$ and assume that the lemma holds true for $r-1$. We then have
	\begin{align*}
	(x+\kappa)^r&=(x+\kappa)\sum_{i=0}^{r-1} \mathfrak{S}_{r-1}^{r-1-i}(\kappa,\si_\q(\kappa),\ldots, \si_\q^{r-2}(\kappa))x^i \\
	&=\sum_{i=0}^{r-1} \si_\q \Bigl( \mathfrak{S}_{r-1}^{r-1-i}(\kappa,\si_\q(\kappa),\ldots, \si_\q^{r-2}(\kappa))\Bigr)x^{i+1} \\
	&\varphantom{=}{}+\sum_{i=0}^{r-1} \de_\q \Bigl( \mathfrak{S}_{r-1}^{r-1-i}(\kappa,\si_\q(\kappa),\ldots, \si_\q^{r-2}(\kappa))\Bigr)x^{i} \\
	&\varphantom{=}{}+\sum_{i=0}^{r-1} \kappa \mathfrak{S}_{r-1}^{r-1-i}(\kappa,\si_\q(\kappa),\ldots, \si_\q^{r-2}(\kappa)) x^{i} \\
	&=\sum_{i=0}^{r-1} \mathfrak{S}_{r-1}^{r-1-i}(\si_\q(\kappa),\si_\q^2(\kappa),\ldots, \si_\q^{r-1}(\kappa)) x^{i+1} \\
	&\varphantom{=}{}+\sum_{i=0}^{r-1} \kappa \si_\q\Bigl( \mathfrak{S}_{r-1}^{r-1-i}(\kappa,\si_\q(\kappa),\ldots, \si_\q^{r-2}(\kappa))\Bigr) x^{i} \\
	&=\sum_{i=1}^{r} \mathfrak{S}_{r-1}^{r-i}(\si_\q(\kappa),\si_\q^2(\kappa),\ldots, \si_\q^{r-1}(\kappa)) x^{i} \\
	&\varphantom{=}{}+\sum_{i=0}^{r-1} \kappa \mathfrak{S}_{r-1}^{r-1-i}(\si_\q(\kappa),\si_\q^2(\kappa),\ldots, \si_\q^{r-1}(\kappa)) x^{i} \\
	&=\sum_{i=0}^{r} \mathfrak{S}_{r}^{r-i}(\kappa, \si_\q(\kappa),\ldots, \si_\q^{r-1}(\kappa)) x^{i}, 
	\end{align*}
	namely, the lemma also holds true for $r$. 
\end{proof}

By virtue of Lemma \ref{lem:h-delta-power}, any element of $E_\q$ can be uniquely expressed as $\sum_ig_i(x+\kappa)^i$ for some $g_i\in R_\q$. Since \eqref{eq:nak-general} can be rewritten as $\nu_\q(x+\kappa)=J(x+\kappa)$, it is easy to verify that $\sum_ig_i(x+\kappa)^i\in E_\q^G$ if and only if $\si_\q(g_i)=J^ig_i$ for all $i$. For the sake of convenience, for any endomorphism $\al$ on a vector space $V$ and any $\lam\in\kk$, denote $\Lam_\lam^\al(V)=\{v \in V \mid \al(v)=\lam v \}$. We obtain

\begin{lemma}\label{lem:inv-quotient}
	For an element $c:=\sum_ig_i(x+\kappa)^i$ in $E_\q$, $c\in E_\q^G$ if and only if $g_i\in \Lam_{J^i}^{\si_\q}(R_\q)$ for all $i$. Moreover, $E_\q^G$ is an $\nan$-graded algebra with $(E_\q^G)_0=R_\q^G$.
\end{lemma}

\begin{proof}
	It is sufficient to prove the last assertion. We have 
	\[
	E_\q^G=\bigoplus_{i\in\nan} \Lam_{J^i}^{\si_\q}(R_\q)(x+\kappa)^i
	\]
	as vector spaces. For any $f\in \Lam_{J^i}^{\si_\q}(R_\q)$ and $g \in \Lam_{J^j}^{\si_\q}(R_\q)$,
	\[
	f(x+\kappa)^i\cdot g(x+\kappa)^j=f\si^i(g)(x+\kappa)^{i+j}.
	\]
	So $E_\q^G$ is graded by the fact $\si\bigl(f\si^i(g)\bigr)=J^if\si^i(J^jg)=J^{i+j}f\si^i(g)$. The equality $(E_\q^G)_0=R_\q^G$ is obvious.
\end{proof}

In order to study equivalence between categories of graded modules over graded algebras, J.J. Zhang introduced a notion of twist algebra \cite{Zhang:twist}. This notion is called Zhang's twist nowadays. Here, we are not going to repeat the original definition of Zhang's twist. Instead, we will introduce a special kind of Zhang's twist which is related with $E_\q^G$. Let $A$ be an $\nan$-graded algebra and $\tau$ a graded automorphism. Define a new multiplication $*_\tau$ on $A$ by $a*_\tau b=a\tau^{\deg a}(b)$ for all homogeneous $a$, $b$.  Then $*_\tau$, which turns out to be associative and preserve the grading, makes $A$ into a new graded algebra. The new algebra is called \textit{Zhang's twist} of $A$. We write $\tw(A,\tau)$ for it.

The external direct sum $\bigoplus_{i\in\nan} \Lam_{J^i}^{\si_\q}(R_\q)$ is a graded algebra via the multiplication on $R_\q$, and we denote it by $\widehat{R_\q}$. Since $\si_\q(\Lam_{J^i}^{\si_\q}(R_\q))=\Lam_{J^i}^{\si_\q}(R_\q)$, $\si_\q$ induces a graded automorphism on $\widehat{R_\q}$. The graded automorphism is denoted by $\widehat{\si_\q}$. By the proof of Lemma \ref{lem:inv-quotient}, we have $E_\q^G\cong\tw(\widehat{R_\q}, \widehat{\si_\q})$.

Next let us pay attention to $E^G$ as well as relations with $E_\q^G$. In a similar procedure, we obtain a graded algebra $\widehat{R}$ as well as a graded automorphism $\widehat{\si}$ induced by $\si$. The situation $\kappa\in R$ is very easy because $E^G$ is a graded subalgebra of $E_\q^G$. From Lemma \ref{lem:inv-quotient}, we conclude

\begin{proposition}
	Suppose that $\kappa\in R$. The invariant $E^G$ is a graded subalgebra of $E_\q^G$, more precisely, $E^G=\bigoplus_i\Lam_{J^i}^{\si}(R)(x+\kappa)^i$ with $(E^G)_0=R^G$. Equivalently, $E^G\cong\tw(\widehat{R}, \widehat{\si})$.
\end{proposition}

The situation $\kappa\notin  R$ is subtle. It is possible that $E^G$ is still a graded subalgebra of $E_\q^G$ or not. We give the following example to indicate the possibility.

\begin{example}
	Let $R=\kk[z_1,z_2]$ and $q\in\kk\setminus\{0,1\}$. Define $\si$ by $\si(z_1)=qz_1$, $\si(z_2)=z_2$, and define $\de$ by $\de(z_1)=q-1$, $\de(z_2)=0$. Then $\kappa=z_1^{-1}\notin R$ and $J=q$.
	\begin{enumerate}
		\item If $q$ is not a root of unity, then $\Lam_{J^i}^{\si}(R)$ is spanned by monomials $z_1^iz_2^j$ for all $j\in\nan$. So $\Lam_{J^i}^{\si}(R)(x+\kappa)^i\subset E^G\cap(E_\q^G)_i$, namely, $E^G$ is a graded subalgebra of $E_\q^G$. Moreover, $E^G\cong\tw(\widehat{R}, \widehat{\si})$.
		\item If $q$ is a primitive $r$th root of unity, then $x^r\in E^G$. However, if we express $x^r$ as the form $\sum_ig_i(x+\kappa)^i$, then  $g_r=1\in \Lam_{J^r}^{\si}(R)$. The fact that $(x+\kappa)^r\notin E$ forces $E^G$ not to be a graded subalgebra of $E_\q^G$.
	\end{enumerate}
\end{example}

Although $E^G$ is not a graded subalgebra of $E_\q^G$ in general, it admits a filtration from the latter. In fact, since $E^G\subset E_\q^G$, we equip $E^G$ with a filtration by $F_iE^G=E^G\cap (E_\q^G)_{\leq i}$. It follows that the associated graded algebra $\gr E^G$ is a graded subalgebra of $E_\q^G\cong\tw(\widehat{R_\q}, \widehat{\si_\q})$. Furthermore, for any $c=\sum_{j=0}^{i}g_j (x+\kappa)^j \in F_iE^G$, we have $g_i\in \Lam_{J^i}^{\si}(R)$. As a consequence, $\gr E^G$ is isomorphic to a graded subalgebra of $\tw(\widehat{R}, \widehat{\si})$. The fact is illustrated schematically:
\[
\xymatrix{
	\gr E^G \ar@{}[rd]|-{\displaystyle\circlearrowright}\ar[r]^-{\mathbf{j}}\ar@{^(->}[d] & \tw(\widehat{R}, \widehat{\si}) \ar@{^(->}[d] \\ E_\q^G \ar[r]^-{\cong} & \tw(\widehat{R_\q}, \widehat{\si_\q})
}
\]
where the map $\mathbf{j}$ is injective. It is natural to ask: Is $\mathbf{j}$ bijective? Or equivalently, do we have $\gr E^G\cong \tw(\widehat{R}, \widehat{\si})$?

The following example gives a negative answer to it.

\begin{example}
	Let $R=\kk[z_1,z_2]$ and $\cha\kk=0$. Define $\si$ by $\si(z_1)=z_1+z_2$, $\si(z_2)=z_2$, and define $\de$ by $\de(z_1)=z_1$, $\de(z_2)=0$. Then $\kappa=z_1z_2^{-1}$ and $J=1$. For all $i\in\nan$, $\tw(\widehat{R}, \widehat{\si})_i=\Lam^\si_1(R)=\kk[z_2]$. We claim that $1\in \tw(\widehat{R}, \widehat{\si})_1$ does not belong to $\im \mathbf{j}$. In fact, if yes, there would be $\vphi \in (E_\q^G)_0$ such that $x+z_1z_2^{-1}+\vphi \in E$. But this is impossible because $(E_\q^G)_0=R_\q^G=\kk(z_2)$.
\end{example}

Finally, let us close this section by a theorem showing when the isomorphism $\gr E^G\cong \tw(\widehat{R}, \widehat{\si})$ holds.

\begin{theorem}\label{thm:gr-invariant}
	Let $R$, $E$, $\nu$ be as above. Then
	\begin{enumerate}
		\item $\gr E^G$ is isomorphic to a graded subalgebra of $\tw(\widehat{R}, \widehat{\si})$,
		\item $\gr E^G$ is isomorphic to $\tw(\widehat{R}, \widehat{\si})$ if either $\nu(x)=Jx$ or $\si$ is of finite order $r$ with $\cha\kk\nmid r$.
	\end{enumerate}
\end{theorem}

\begin{proof}
	Assertion (1) has been proven before. Let us prove assertion (2) now. 
	
	The first case: $\nu(x)=Jx$. For any homogeneous $g_i\in \tw(\widehat{R}, \widehat{\si})_i$, we have $g_i\in \Lam_{J^i}^{\si}(R)$. It follows that $g_ix^i\in E^G$ and $g_ix^i=g_i(x+\kappa)^i+\text{lower terms} \in F_iE^G$, and we have $\gr E^G\cong \tw(\widehat{R}, \widehat{\si})$.
	
	The second case: $\si$ is of finite order $r$ with $\cha\kk\nmid r$. In this case, $\si$ is semisimple (i.e., diagonalizable). Choose any $r$th primitive root $\zeta$ of unity, and hence we have $R=\bigoplus_{j\in \inn/r\inn}\Lam^\si_{\zeta^j}(R)$. For the same reason, $E=\bigoplus_{j\in \inn/r\inn}\Lam^\nu_{\zeta^j}(E)$. Clearly, $J=\zeta^s$ for a unique $s \in \inn/r\inn$.
	
	For any homogeneous $g_i\in \tw(\widehat{R}, \widehat{\si})_i$, we decompose $g_ix^i$ as $\sum_{j\in \inn/r\inn}p_j$ with $p_j \in \Lam^\nu_{\zeta^j}(E)$. All $p_j$'s are expressed as polynomials in $x$ with coefficients in $R$. Denote by $l$ the highest degree of these polynomials. Namely,
	\[
	p_j=f_{l,j}x^{l}+f_{l-1,j}x^{l-1}+\cdots+f_{0,j}
	\]
	with $f_{\bullet,\bullet}\in R$ and at least one $f_{l,j}\neq 0$. Obviously, $l\geq i$. Since $p_j \in \Lam^\nu_{\zeta^j}(E)$, and
	\begin{align*}
	\nu(p_j)&=J^l\nu(f_{l,j})x^l+\text{lower terms}, \\
	\zeta^jp_j&=\zeta^j f_{l,j}x^l+\text{lower terms},
	\end{align*}
	we have $J^l\nu(f_{l,j})=\zeta^j f_{l,j}$. It follows that $\si(f_{l,j})=\zeta^{ls-j}f_{l,j}$, i.e., $f_{l,j}\in \Lam^\si_{\zeta^{ls-j}}(R)$. If $l>i$, then $0=\sum_j f_{l,j}\in \bigoplus_{j\in \inn/r\inn}\Lam^\si_{\zeta^{ls-j}}(R)=\bigoplus_{j\in \inn/r\inn}\Lam^\si_{\zeta^j}(R)$. Thus all $f_{l,j}$'s are equal to zero, a contradiction. So $l=i$, and together with the fact $g_i\in \Lam^\si_{is}(R)$, we conclude that $f_{i,0}=g_i$ and $f_{i,j}=0$ for all $1\leq j\leq r$. Therefore,
	\[
	p_0=f_{i,0}x^{i}+f_{i-1,0}x^{i-1}+\cdots+f_{0,0}=g_i(x+\kappa)^i+\text{lower terms}
	\] 
	which belongs to $F_iE^G$. Consequently, we have $\gr E^G\cong \tw(\widehat{R}, \widehat{\si})$.
\end{proof}

\section*{Acknowledgments}
The authors acknowledge the supports of the Natural Science Foundation of China No.\ 11501492 and No.\ 11711530703. L. L. would like to thank Can Zhu for giving a talk based on his joint paper \cite{Zhu-Oystaeyen-Zhang:double-ore-extension} with F.~Van Oystaeyen and Y.~Zhang.


\begin{thebibliography}{10}
	
	\bibitem{Bell-Zhang:Zariski-cancellation}
	J.~Bell and J.J. Zhang, \emph{Zariski cancellation problem for noncommutative
		algebras}, Selecta Math. (N.S.) \textbf{23} (2017), 1709--1737.
	
	\bibitem{Brown-Gilmartin:qhs-connectedHopf}
	K. Brown and P. Gilmartin,  \emph{Quantum homogeneous spaces of connected Hopf algebras}, J. Algebra \textbf{454} (2015), 400--432.
	
	\bibitem{Brown-Zhang:AS-Gorenstein-Hopf}
	K.A. Brown and J.J. Zhang, \emph{Dualising complexes and twisted {H}ochschild
		(co)homology for {N}oetherian {H}opf algebras}, J. Algebra \textbf{320}
	(2008), 1814--1850.
	
	\bibitem{Ginzburg:CY-alg}
	V.~Ginzburg, \emph{Calabi-{Y}au algebras}, preprint, arXiv:math/0612139, 2006.
	
	\bibitem{Goodman-Krahmer:untwist-cy}
	J.~Goodman and U.~Kr{\"a}hmer, \emph{Untwisting a twisted {C}alabi-{Y}au
		algebra}, J. Algebra \textbf{406} (2014), 272--289.
	
	\bibitem{He-Oystaeyen-Zhang:Koz-AS-reg}
	J.-W. He, F.~Van~Oystaeyen, and Y.~Zhang, \emph{Skew polynomial algebras with
		coefficients in {K}oszul {A}rtin-{S}chelter {G}orenstein algebras}, J.
	Algebra \textbf{390} (2013), 231--249.
	
	\bibitem{Jorgensen-Zhang:gorenstein}
	P.~J{\o}rgensen and J.J. Zhang, \emph{Gourmet's guide to {G}orensteinness}, Adv.
	Math. \textbf{151} (2000), 313--345.
	
	\bibitem{Krahmer:qh-space}
	U.~Kr{\"a}hmer, \emph{On the {H}ochschild (co)homology of quantum homogeneous
		spaces}, Israel J. Math. \textbf{189} (2012), 237--266.
	
	\bibitem{L-Wu-Wang:twisted-CY-Ore-extension}
	L.~Liu, S.~Wang, and Q.~Wu, \emph{Twisted {C}alabi-{Y}au property of {O}re
		extensions}, J. Noncommut. Geom. \textbf{8} (2014), 587--609.
	
	\bibitem{Lv-Mao-Zhang:NAK}
	J.-F. L\"u, X.-F. Mao, and J.J. Zhang, \emph{Nakayama automorphism and
		applications}, Tran. Amer. Math. Soc. \textbf{369} (2017), 2425--2460.
	
	\bibitem{Lv-Mao-Zhang:Nak-graded}
	J.-F. L{\"u}, X.-F. Mao, and J.J. Zhang, \emph{Nakayama automorphisms of a
		class of graded algebras}, Israel J. Math. \textbf{219} (2017), 707--725.
	
	\bibitem{McConnell-Robson:noncomm-ring}
	J.C. McConnell and J.C. Robson, \emph{Noncommutative noetherian rings}, John
	Wiley \& Sons Ltd., Chichester, 1987.
	
	\bibitem{Reyes-Rogalski-Zhang:skew-CY-homo-identity}
	M.~Reyes, D.~Rogalski, and J.J. Zhang, \emph{Skew {C}alabi-{Y}au algebras and
		homological identities}, Adv. Math. \textbf{264} (2014), 308--354.
	
	\bibitem{Shen-Lu:Nak-PBW}
	Y.~Shen and D.~Lu, \emph{Nakayama automorphisms of {PBW} deformations and
		{H}opf actions}, Sci. China Math. \textbf{59} (2016), 661--672.
	
	\bibitem{Yekutieli:dual-comp}
	A.~Yekutieli, \emph{Dualizing complexes over noncommutative graded algebras},
	J. Algebra \textbf{153} (1992), 41--84.
	
	\bibitem{Yekutieli:rigid-dualizing-complex-universal-enveloping-algebra}
	\bysame, \emph{The rigid dualizing complex of a universal enveloping algebra},
	J. Pure Appl. Algebra \textbf{150} (2000), 85--93.
	
	\bibitem{Zhang:twist}
	J.J. Zhang, \emph{Twisted graded algebras and equivalences of graded
		categories}, Proc. London Math. Soc. (3) \textbf{72} (1996), 281--311.
	
	\bibitem{Zhu-Oystaeyen-Zhang:double-ore-extension}
	C.~Zhu, F.~Van~Oystaeyen, and Y.~Zhang, \emph{Nakayama automorphisms of double
		{O}re extensions of {K}oszul regular algebras}, Manuscripta Math.
	\textbf{152} (2017), 555--584.
	
\end{thebibliography}

\end{document}